\documentclass{amsart}
\usepackage{setspace}
\usepackage{a4}
\usepackage{amsthm}
\usepackage{latexsym}
\usepackage{amsfonts}
\usepackage{graphicx}
\usepackage{textcomp}
\usepackage{cite}
\usepackage{enumerate}
\usepackage{amssymb}
\usepackage{hyperref}
\usepackage{amsmath}
\usepackage{tikz}
\usepackage[mathscr]{euscript}
\usepackage{mathtools}
\newtheorem{theorem}{Theorem}[section]

\newtheorem{definition}[theorem]{Definition}

\newtheorem{problem}[theorem]{Problem}

\title{This is the title}
\usepackage{fancyhdr}

\begin{document}
\hrule\hrule\hrule\hrule\hrule
\vspace{0.3cm}	
\begin{center}
{\bf{Continuous Donoho-Elad Spark Uncertainty Principle}}\\
\vspace{0.3cm}
\hrule\hrule\hrule\hrule\hrule
\vspace{0.3cm}
\textbf{K. MAHESH KRISHNA}\\
School of Mathematics and Natural Sciences\\
Chanakya University Global Campus\\
NH-648, Haraluru Village\\
Devanahalli Taluk, 	Bengaluru  North District\\
Karnataka State 562 110 India\\
Email: kmaheshak@gmail.com\\

Date: \today
\end{center}

\hrule\hrule
\vspace{0.5cm}
\textbf{Abstract}:  Donoho and Elad \textit{[Proc. Natl. Acad. Sci. USA, 2003]} introduced the important notion of the spark of a frame, using which they derived a fundamental uncertainty principle. Based on spark, they also provided a necessary and sufficient condition for the uniqueness of sparse solutions to the NP-hard $\ell_0$-minimization problem. In this nano note, we show that the notion of spark can be extended to linear maps whose domains are measure spaces. Using this generalization, we derive an uncertainty principle and provide a sufficient condition for the existence of sparse solutions to linear systems on measure spaces.

\textbf{Keywords}: Frame, Uncertainty principle, Sparse solution.

\textbf{Mathematics Subject Classification (2020)}: 94A12, 42C15, 94A08, 28A05.\\

\hrule

\hrule
\section{Introduction}
Let $\mathcal{H}$ be a  finite dimensional Hilbert  space over $\mathbb{K}$ ($\mathbb{C}$ or $\mathbb{R}$). Recall that \cite{BEDETTOFICKUS} a collection of nonzero elements $\{\tau_j\}_{j=1}^n$ in $\mathcal{H}$ is said to be a \textbf{frame} (also known as \textbf{dictionary}) $\mathcal{H}$ if there are $r, s >0$ such that 
\begin{align*}
	r\|h\|^2\leq \sum_{j=1}^{n}|\langle h,  \tau_j \rangle|^2\leq s \|h\|^2, \quad \forall h \in \mathcal{H}.
\end{align*}
It is well-known that a collection $\{\tau_j\}_{j=1}^n$ in $\mathcal{H}$  is a frame for $\mathcal{H}$  if and only if $\{\tau_j\}_{j=1}^n$ spans $\mathcal{H}$ \cite{HANKORNELSON}.  A frame $\{\tau_j\}_{j=1}^n$ for  $\mathcal{H}$ is said to be \textbf{normalized} if $\|\tau_j\|=1$ for all $1\leq j \leq n$. Note that any frame can be normalized by dividing each element by its norm.
Given a frame $\{\tau_j\}_{j=1}^n$ for  $\mathcal{H}$, we define the analysis operator
\begin{align*}
	\theta_\tau:\mathcal{H}\ni h \mapsto \theta_\tau h\coloneqq(\langle h,  \tau_j \rangle)_{j=1}^n \in \mathbb{K}^n.
\end{align*}
Adjoint of the analysis operator is known as the synthesis operator whose equation is 
\begin{align*}
	\theta_\tau^*: \mathbb{K}^n \ni (a_j)_{j=1}^n \mapsto \theta_\tau^*(a_j)_{j=1}^n\coloneqq\sum_{j=1}^{n}a_j\tau_j \in \mathcal{H}.
\end{align*}
Given $d\in \mathbb{K}^n$, let $\|d\|_0$ be the number of nonzero entries in $d$. Central problem which occurs in many situations    is the following $\ell_0$-minimization problem: 
\begin{problem} \label{P0}
	Let $\{\tau_j\}_{j=1}^n$ be a normalized frame for  $\mathcal{H}$.  Given $h \in \mathcal{H}$, solve 	
	\begin{align*}
		\operatornamewithlimits{minimize}_{d\in \mathbb{K}^n}\|d\|_0 \quad \text{ subject to } \quad \theta_\tau^*d=h.
	\end{align*}
\end{problem}
Recall that $c \in \mathbb{K}^n$ is said to be a unique solution to Problem \ref{P0} if it satisfies following two conditions.
\begin{enumerate}[\upshape(i)]
	\item $\theta_\tau^*c=h$.
	\item If $d \in \mathbb{K}^n$ satisfies $\theta_\tau^*d=h$, then 
	\begin{align*}
		\|d\|_0>\|c\|_0.
	\end{align*}
\end{enumerate}
In 1995, Natarajan showed that Problem \ref{P0} is NP-Hard \cite{NATARAJAN}. Therefore solution to Problem \ref{P0} has to be obtained using other methods. Work which is built around Problem \ref{P0} is known as \textbf{sparseland}  (term due to Elad \cite{ELAD}) or \textbf{compressive sensing} or \textbf{compressed sensing}.

As the operator $\theta_\tau^*$ is surjective, for a given $h \in \mathcal{H}$, there is always a  $d\in \mathbb{K}^n$ such that  $\theta_\tau^*d=h.$ Thus the central problem is when solution to Problem \ref{P0} is unique. One of the greatest results of Donoho and Elad \cite{DONOHOELAD} with regard to this is using the notion of spark defined as follows. In the paper, given a subset $M\subseteq \mathbb{N}$, the cardinality of $M$ is denoted by $o(M)$.
\begin{definition}\cite{DONOHOELAD}\label{SPARK}
	Given a  normalized frame $\{\tau_j\}_{j=1}^n$ for  $\mathcal{H}$, the \textbf{spark} of 	$\{\tau_j\}_{j=1}^n$ is defined as 
	\begin{align*}
		\operatorname{Spark}(\{\tau_j\}_{j=1}^n)&\coloneqq \min\{o(M): M \subseteq \{1, \dots, n\}, \{\tau_j\}_{j\in M} \text{ is linearly dependent}\}\\
		&=\min\{\|d\|_0:d \in \ker(\theta_\tau^*), d \neq 0\}.
	\end{align*}
\end{definition}
In 2003,  Donoho and Elad derived the following breakthrough spark uncertainty principle \cite{DONOHOELAD}.
\begin{theorem} \cite{DONOHOELAD} \label{DESU} (\textbf{Donoho-Elad Spark Uncertainty Principle}) 
	Let $\{\tau_j\}_{j=1}^n$ be a   normalized frame for  $\mathcal{H}$. 	If $a, b \in \mathbb{K}^n$ are distinct and $\theta _\tau^* a=\theta_\tau^* b $, then 
	\begin{align*}
		\|a\|_0+\|b\|_0 \geq \operatorname{Spark}(\{\tau_j\}_{j=1}^n).
	\end{align*}
\end{theorem}
In the same paper \cite{DONOHOELAD}, Donoho and Elad also gave a characterization for the solution of Problem  \ref{P0} using spark. 
\begin{theorem}\cite{DONOHOELAD, DAVENPORTDUARTEELDARKUTYNIOK}\label{DE} (\textbf{Donoho-Elad Spark Sparsity Theorem})
	Let $\{\tau_j\}_{j=1}^n$ be a   normalized frame for  $\mathcal{H}$. 
	\begin{enumerate}[\upshape(i)]
		\item For every $h \in \mathcal{H}$ and every $1\leq k \leq n$, there exists atmost one vector $c\in \mathbb{K}^n$ such that 
		\begin{align*}
			h=\theta_\tau^*c \quad \text{satisfying} \quad \|c\|_0\leq k
		\end{align*}
		if and only if 
		\begin{align*}
			\operatorname{Spark}(\{\tau_j\}_{j=1}^n)>2k.	
		\end{align*}
		\item If $h \in \mathcal{H}$ can be written as  $	h=\theta_\tau^*c$ for some  $c\in \mathbb{K}^n$ satisfying 
		\begin{align*}
			\|c\|_0<\frac{1}{2}\operatorname{Spark}(\{\tau_j\}_{j=1}^n),
		\end{align*}	
		then $c$ is the unique solution to Problem \ref{P0}.
	\end{enumerate}
\end{theorem}
In this note, we show that Definition \ref{SPARK} can be extended largely. Using this, we show that Theorems \ref{DESU} and \ref{DE} have  continuous extensions.

\section{Continuous  Spark}
	Let $(\Omega, \mu)	$ be a measure space and let 
	\begin{align*}
		\mathcal{M}(\Omega, \mu)\coloneqq \{f: \Omega \to \mathbb{K} \text{ is measurable}\}.
	\end{align*}
Let $\mathcal{V}$ be a vector space over $\mathbb{K}$ and let $\mathcal{W}$ be a subspace of $\mathcal{M}(\Omega, \mu)$. Given a linear map $A:\mathcal{W} \to \mathcal{V}$, we define the spark of $A$ as 
\begin{align*}
	\operatorname{Spark}(A)	\coloneqq \inf\{\mu(\operatorname{supp}(f)): f \in \ker(A), f \neq 0\}.
\end{align*}
We now have continuous  version of Theorem \ref{DESU}.
\begin{theorem} (\textbf{Continuous  Donoho-Elad Spark Uncertainty Principle})
Let $A:\mathcal{W} \to \mathcal{V}$ be a linear map.	If $f, g  \in \mathcal{W}$ are distinct and $Af=Ag $, then 
\begin{align*}
	\mu(\operatorname{supp}(f))+ \mu(\operatorname{supp}(g)) \geq \operatorname{Spark}(A).
\end{align*}
\end{theorem}
\begin{proof}
	Since $f-g \neq 0$ and $f-g \in \ker(A)$, we have 
	\begin{align*}
	\operatorname{Spark}(A) \leq \mu(\operatorname{supp}(f-g))\leq \mu(\operatorname{supp}(f) \cup \operatorname{supp}(g)) 
	\leq \mu(\operatorname{supp}(f)) + \mu(\operatorname{supp}(g)).
	\end{align*}
\end{proof}
We set of most general version of Problem \ref{P0}  as follows. 
\begin{problem}\label{CP0}
Let $A:\mathcal{W} \to \mathcal{V}$ be a linear map.  Given $v \in \mathcal{V}$, solve 	
	\begin{align*}
		\operatornamewithlimits{minimize}_{g \in \mathcal{W}}\mu(\operatorname{supp}(g)) \quad \text{ subject to } \quad Ag=v.
	\end{align*}
\end{problem}
Following is continuous  version of Theorem \ref{DE}.
\begin{theorem}\label{ADE} (\textbf{Continuous Donoho-Elad Spark Sparsity Theorem})
Let $A:\mathcal{W} \to \mathcal{V}$ be a linear map.
	\begin{enumerate}[\upshape(i)]
		\item Let $r \in [0, \infty).$ If
		\begin{align*}
			\operatorname{Spark}(A)>2r, 	
		\end{align*}
	then for every $v \in \mathcal{V}$, there exists atmost one vector $f\in \mathcal{W}$ such that 
	\begin{align*}
		v=Af \quad \text{satisfying} \quad \mu(\operatorname{supp}(f))\leq r.
	\end{align*}
		\item If $v \in \mathcal{V}$ can be written as  $	v=Af$ for some  $f\in \mathcal{W}$ satisfying 
		\begin{align*}
			\mu(\operatorname{supp}(f))<\frac{1}{2}\operatorname{Spark}(A),
		\end{align*}	
		then $f$ is the unique solution to Problem \ref{CP0}.
	\end{enumerate}
\end{theorem}
\begin{proof}
	\begin{enumerate}[\upshape(i)]
		\item 	Let $r \in [0, \infty)$ and  $v \in \mathcal{V}$.	Let $g, h\in \mathcal{W}$ satisfy 	$v=Ag=Ah$ and $\mu(\operatorname{supp}(g))\leq r$,  $\mu(\operatorname{supp}(h))\leq r$.    We claim that $g=h$. If this is not true, then $g-h\neq 0$. Then   $ g-h\in \ker(A)$ with $g-h\neq 0$. But then 
		\begin{align*}
			2r&<\operatorname{Spark}(A)\leq \mu(\operatorname{supp}(g-h))\leq \mu(\operatorname{supp}(g) \cup \operatorname{supp}(h))\\
			&\leq \mu(\operatorname{supp}(g)) + \mu(\operatorname{supp}(h))\leq r+r=2r
		\end{align*}
		which is impossible. Hence claim holds.
		 	\item Let $v \in \mathcal{V}$ and  $f \in \mathcal{W}$ satisfies 	
		 		\begin{align*}
		 		\mu(\operatorname{supp}(f))<\frac{1}{2}\operatorname{Spark}(A).
		 	\end{align*}	
		 	Let $g\in \mathcal{W}$ be such that  	$v=Ag$ and $g\neq f$. 
		 		 Then we have 
		\begin{align*}
			A(f-g)=v-v=0.
		\end{align*}
		Hence $f-g \in \ker(A)$ and $f-g\neq 0$. Definition of  spark then gives
		\begin{align*}
			\operatorname{Spark}(A)&\leq \mu(\operatorname{supp}(f-g))\leq \mu(\operatorname{supp}(f))+ \mu(\operatorname{supp}(g))\\
			&<\frac{1}{2}\operatorname{Spark}(A)+\mu(\operatorname{supp}(g)).
		\end{align*}
		Therefore 
		\begin{align*}
			\mu(\operatorname{supp}(f))<\frac{1}{2}\operatorname{Spark}(A)<\mu(\operatorname{supp}(g)).	
		\end{align*}
		Hence $f$ is unique solution to Problem \ref{CP0}.
	\end{enumerate}
\end{proof}
In view of Theorem \ref{DE}, we have following problem: For which measure spaces, the converse of (i) in Theorem  \ref{ADE} holds? Note that the proof of converse of (i) in Theorem  \ref{DE} is based on the technique of writing a $2k$-sparse vector as a difference of two $k$-sparse vectors \cite{DAVENPORTDUARTEELDARKUTYNIOK} which we are unable do in continuous setting.

 \bibliographystyle{plain}
 \bibliography{reference.bib}

\begin{thebibliography}{1}

\bibitem{BEDETTOFICKUS}
John~J. Benedetto and Matthew Fickus.
\newblock Finite normalized tight frames.
\newblock {\em Adv. Comput. Math.}, 18(2-4):357--385, 2003.

\bibitem{DAVENPORTDUARTEELDARKUTYNIOK}
Mark~A. Davenport, Marco~F. Duarte, Yonina~C. Eldar, and Gitta Kutyniok.
\newblock Introduction to compressed sensing.
\newblock In {\em Compressed sensing}, pages 1--64. Cambridge Univ. Press,
  Cambridge, 2012.

\bibitem{DONOHOELAD}
David~L. Donoho and Michael Elad.
\newblock Optimally sparse representation in general (nonorthogonal)
  dictionaries via {$l^1$} minimization.
\newblock {\em Proc. Natl. Acad. Sci. USA}, 100(5):2197--2202, 2003.

\bibitem{ELAD}
Michael Elad.
\newblock {\em Sparse and redundant representations: From theory to
  applications in signal and image processing}.
\newblock Springer, New York, 2010.

\bibitem{HANKORNELSON}
Deguang Han, Keri Kornelson, David Larson, and Eric Weber.
\newblock {\em Frames for undergraduates}, volume~40 of {\em Student
  Mathematical Library}.
\newblock American Mathematical Society, Providence, RI, 2007.

\bibitem{NATARAJAN}
B.~K. Natarajan.
\newblock Sparse approximate solutions to linear systems.
\newblock {\em SIAM J. Comput.}, 24(2):227--234, 1995.

\end{thebibliography}

\end{document}